\documentclass[12pt]{article}
\usepackage[left=3cm,right=3cm,
    top=2cm,bottom=2cm,bindingoffset=0cm]{geometry}
\usepackage{graphicx}
\usepackage[english]{babel}
\usepackage{amsmath,amscd, wasysym, mathtools}
\usepackage{upgreek}
\usepackage{hyperref} 
\usepackage[all]{xy}
\usepackage{tikz-cd}
\usepackage{amssymb,amsthm,euscript}
\usepackage{xcolor,graphicx,epsfig,subfigure}
\usepackage{dsfont}
\usepackage{enumitem}

\usepackage{amsfonts}
\usepackage{fancyhdr}
\newtheorem{theorem}{Theorem}
\newtheorem*{theorem*}{Theorem}
\newtheorem{lemma}{Lemma}
\newtheorem{proposition}{Proposition}

\newtheorem{fact}{Fact}
\newtheorem{corollary}{Corollary}

\theoremstyle{definition}
\newtheorem{definition}{Definition}

\newtheorem{remark}{Remark}
\renewenvironment{proof}{
\par\noindent{\it Proof.}} {\mbox{}\hfill$\blacksquare$ \par }
\newcommand{\iso}{\xrightarrow{\sim}}
\newcommand{\into}{\hookrightarrow}
\newcommand{\bbP}{\mathbb{P}}

\newcommand{\bbZ}{\mathbb{Z}}
\newcommand{\bbQ}{\mathbb{Q}}
\newcommand{\algQ}{\overline{\mathbb{Q}}}
\newcommand{\algK}{\overline{K}}

\newcommand {\Gal}{\mathord{\rm {Gal}}}
\newcommand{\Aut}{\mathord{\rm{Aut}}}
\newcommand{\End}{\mathord{\rm{End}}}
\newcommand{\Spec}{\mathord{\rm{Spec}}}
\newcommand{\reg}{\mathord{\rm{reg}}}
\newcommand{\Sch}{\mathbf{Sch}}
\newcommand{\res}{\mathord{\rm{res}}}
\newcommand{\id}{\mathord{\rm{id}}}
\newcommand{\C}{\mathord{\rm{C}}}

\DeclareFontFamily{U}{wncy}{}
    \DeclareFontShape{U}{wncy}{m}{n}{<->wncyr10}{}
    \DeclareSymbolFont{mcy}{U}{wncy}{m}{n}
    \DeclareMathSymbol{\Sh}{\mathord}{mcy}{"58}

\title{On the fields of definition of genus-one covers of $\mathbb{P}^1$}
\author{Alexander Molyakov}
\date{}

\begin{document}
\maketitle
\begin{abstract}
It is known that sometimes a Belyi pair is not defined over its field of moduli. Instead, it is defined over a finite degree extension of its field of moduli, called a field of definition. We show that given a number $m$ there exists a Belyi pair such that the degree of a field of definition over the field of moduli is greater than~$m$.  As a byproduct, we obtain a counterexample to the local-global principle for Belyi pairs.
\end{abstract}

\let\thefootnote\relax\footnotetext{This is the accepted version of the article \textit{On the fields of definition of genus-one covers of  $\bbP^1$}, Bulletin of the London Mathematical Society, which has been published in final form at \url{http://doi.org/10.1112/blms.12840} }

\section{Introduction}
The theory of covers of the projective line has been rapidly developing over the last forty years. Covers unramified outside $\{0,1,\infty\}$ are called \textit{Belyi functions}, they are of special interest due to applications to the inverse Galois problem. The famous Belyi theorem asserts that a complex curve $X$ is defined over $\overline{\mathbb{Q}}$ iff there exists a Belyi function $\beta:X\to\bbP^1$ \cite{B79}. The category of \textit{Belyi pairs} $(X,\beta)$ is equivalent to the category of \textit{dessins d'enfants}, graphs embedded into surfaces in a particular way \cite[p.188, Theorem]{S17}. Thus the natural Galois-action on Belyi pairs induces a Galois-action on the set of dessins, which turns out to be exact. 

\textit{The absolute field of moduli} of a (ramified) cover $X\to \bbP^1$ is the fixed field of its stabilizer with respect to the natural action of the absolute Galois group $\Gal ( \algQ / \bbQ)$. Clearly it is contained in each field of definition. Covers without automorphisms as well as regular covers are known to be defined over their field of moduli \cite{DD97}. In fact, covers defined over the field of moduli form a Zariski-dense subset of the corresponding Hurwitz space \cite{DDM04}. However, there are examples of covers (and even of Belyi pairs) not defined over the field of moduli \cite{C94, CG94, SF95}. The examples presented in these papers give us covers which are defined over quadratic extensions of the corresponding field of moduli. It turns out that any cover of genus 0 is defined over a quadratic extension \cite{C94}. We deal with the genus-one case where, as we show, the degree of a minimal field of definition over the field of moduli can be arbitrary large. Precisely, we prove the following result which appears as Theorem \ref{th2} below:

\begin{theorem*} Let $E$ be a non-CM elliptic curve over $\algQ$ and let $m$ be a positive integer. Then there exists a Belyi function $\beta:E \to \bbP^1_{\algQ}$  with the absolute field of moduli $M^{abs}$ and the following property:
for any field of definition $L$ of the Belyi pair $(E,\beta)$, we have $[L:M^{abs}]\geq m$.
\end{theorem*}

The construction proceeds as follows: using the idea  of a reduced cover \cite{DE99} models of the initial cover can be put in one-to-one correspondence with models of the regular cover $E\to E/\Aut(E,\beta)$. The regular cover turns out to be an isogeny in the case we are interested. Afterwards some analysis of cohomological obstructions demonstrates the existence of an isogeny with prescribed properties. Finally, the refined version of Belyi theorem from \cite{C96} is applied to complete the construction. 

The cohomological obstructions for a field of moduli being a field of definition was introduced by P. Dèbes and J.-C. Douai in a more general context of arbitrary algebraic covers \cite{DD97}. This obstruction consists of a family of elements in $H^2(G_M,\C(\Aut (X,f)))$ where $G_M$ is the absolute Galois group of a field of moduli $M$ and $\C(\Aut (X,f))$ stands for the centre of the  automorphisms group of a cover $(X,f)$. However, even in the case of covers with an abelian group of automorphisms when the obstruction reduces to a single element, it is generally unclear which elements in the cohomology group arise in this way. This makes constructing examples by means of purely cohomological techniques more difficult. The key observation in the genus-one case is that under some assumptions each element of $H^2(G_M,E[m])$ appears as an obstruction for a Belyi pair (Lemma \ref{lemma3}). 

Our approach can be also applied to produce counterexamples to the Hasse principle for Belyi pairs in the following sense (see Corollary \ref{cor4} below):

\begin{theorem*} There exists a genus-one Belyi pair $(X,f)$ over $\algQ$ which cannot be defined over its field of moduli $M$, yet which can be defined over the completion $M_\nu$ for all places $\nu$ of $M$.\end{theorem*}

Whereas the local-to-global problem has a positive answer for $G$-covers with an exception of the special case of Grunwald–Wang theorem \cite{DD97, DD98}, the Hasse principle fails to be true for mere covers. Though a counterexample is presented in \cite{CR04}, counterexamples for Belyi pairs seem to be absent in the literature. The method we use to obtain these examples does not immediately lead to explicit formulas but it is effective and appears to be computable in practice.   

I would like to thank G. Shabat and other participants of the seminar "Graphs on surfaces and curves over number fields" for helpful discussions, fruitful remarks and encouraging interest in the subject. I am also grateful to J.-M. Couveignes for his comments that helped to improve the exposition and to S. Gorchinskiy for pointing out to me the folklore fact which is stated as Proposition \ref{prop5} below.

\section{Preliminaries}
\subsection{Base change and descent}
\textbf{Notation.} Let $\Sch_L$ be the category of schemes over a field $L$. 

To an automorphism $g:L\to L$ we can associate a \textit{base change functor} $g_*:\Sch_L\to \Sch_L$ which is defined as follows:
for an $L$-scheme $X$ with a structure morphism $\psi\colon X\to \Spec(L)$, the object $g_*X\in \Sch_L$ is the scheme $X$ with a modified structure morphism $(g^*)^{-1}\circ \psi\colon X\to \Spec(L)$ where $g^*\colon \Spec(L)\to\Spec(L)$ is the automorphism induced by $g$. In other words, $g_*$ is the pullback functor $\Sch_L\to \Sch_L$ under the map $g^*\colon \Spec(L)\to\Spec(L)$.

If $X\in \Sch_K$ is a scheme given over a smaller field $K\subset L$, then $X_L$ stands for the $L$-scheme $X\times_{\Spec (K)} \Spec (L)$, that is the \textit{extension of constants}. By abuse of notations we also write $X_K$ for a $K$-model of a scheme $X$ over the larger field $L$ when it leads to no confusion.  

We recall that when $L\supset K$ is a Galois extension, $X_L$ admits a \textit{descent datum} consisting of a family of $L$-isomorphisms $\rho_X(g):g_{*}X_L\to X_L$ for each $g\in\Gal(L/K)$ satisfying the cocycle condition $\rho_X(gh)=\rho_X(g)\circ g_*\rho_X (h)$ \cite[(14.17.1)]{GW}. A morphism of $L$-schemes with descent data is an $L$-morphism $f: X\to Y$ \textit{compatible with the descent data} in the sense that $f\circ \rho_X(g)=\rho_Y (g)\circ g_*f$.
The category of quasi-projective varieties admits an effective Galois descent \cite[Theorem 14.86]{GW}. This fact can be stated as follows:

\begin{fact}\label{fact1} Let $K\subset L$ be a finite Galois extension. Then the category of quasi-projective varieties over $L$ with descent data is equivalent to the category of quasi-projective varieties over the smaller field $K$ where the functor in one direction is an extension of constants.\end{fact}

For another treatment of Galois descent in the classical language of algebraic varieties see \cite[Corollary 16.25]{Mag} and the foundational paper of Weil \cite{Weil}.

\subsection{Covers of curves and their fields of moduli}

We will use the following notations and assumptions throughout this paper: all curves are assumed to be smooth, proper and geometrically integral; we also fix a base field $K$ of characteristic zero and a Galois extension $F\supset K$; all fields are assumed to contain $K$ and to be contained in $F$. 

\begin{definition}Consider a curve $B$ over a field $L$, $K\subset L\subset F$.
\begin{enumerate} [label=(\roman*)] 
    \item \textit{The category of (ramified) covers} of $B$ consists of pairs $(X,f)$ where $f:X\to B$ is a surjective $L$-morphism of curves over $L$. 
\item \textit{A morphism of covers} $h:(X_1,f_1)\to (X_2,f_2)$ is a morphism $h:X_1\to X_2$ such that $f_2\circ h= f_1$. In particular, \textit{an automorphism of a cover} $f:X\to B$ is an automorphism $\sigma\in \Aut(X)$ such that $f\circ \sigma=f$. \end{enumerate}
\end{definition}

There is a natural action of $\Gal(F/L)$ on the category of covers $f:X\to B_F$. It is defined as follows: $$g_*(X,f)=(g_*X, \rho_B(g)\circ g_*(f)).$$ Note that this action depends on the chosen $L$-model  $B$. 

In the following proposition a \textit{pointed cover} is a surjective morphism of curves $f\colon X\to B$ with marked rational points $Q\in X,\, P\in B$ such that $f(Q)=P$.
\begin{proposition} \label{prop1}
Consider a finite Galois extension $M\subset F$, a curve $B$ over $M$ and a~(pointed) cover $f:X\to B_F$. Then the cover $(X,f)$ admits a~(pointed) $M$-model iff there exists a family of (pointed) cover isomorphisms $\rho(g):\; g_*(X,f)\to (X,f)$ defined for each $g\in \Gal(F/M)$ and satisfying the cocycle condition \begin{equation}\label{1}
     \rho(gh)=\rho(g)\circ g_*\rho(h).\end{equation}
Furthermore, any such family of isomorphisms $\{\rho(g)\}_g$ satisfying \eqref{1} defines a unique (pointed) $M$-model of the cover $(X,f)$.
\end{proposition}
\begin{proof}
This statement follows form the effective Galois descent for the category of quasi-projective varieties (Fact \ref{fact1}) since all curves are quasi-projective. A rational point on a curve over $F$ can be viewed as a morphism from the one-point scheme $\Spec (F)$. Thus, in the case of pointed covers it suffices to apply Galois descent to the following diagram: 
\begin{equation*}
\xymatrix{
& X\ar[d]^{f}\\
\Spec(F)\ar[ru]^Q\ar[r]^{\;\;\;\;P} & B_F
}
\end{equation*}

\end{proof}

\begin{definition}
Let $B$ be a curve over $K$. \begin{enumerate}[label=(\roman*)] 
\item {\textit {The absolute field of moduli} of a cover $f:X \to B_{F}$ is the fixed field $F^{\textrm{Stab}(f)}$ where the stabilizer is taken with respect to the above action of $\Gal(F/K)$ on the set of equivalence classes of covers. We also use the notion of \textit{a relative field of moduli} for any field containing the absolute field of moduli (``a field of moduli'' will always refer to this case).}
\item {We say that $f$ \textit{is defined over a field $L\subset F$} if there exists an $L$-model $f_L:X_L\to B_L$ isomorphic to $(X,f)$ over $F$. In this case $L$ is called \textit{a field of definition.}}
\end{enumerate}\end{definition} 

In other words a field $M$ is said to be a field of moduli of a cover $f:X\to B_{F}$ if $(X,f)$ is isomorphic to $g_*(X,f)$ for each $g\in\Gal(F/M)$.

The two concepts we have just introduced are in the following relation as was demonstrated in \cite{DD97}: 

\begin{proposition}\label{prop2}
\begin{enumerate}[label=(\roman*)]
\item{Every field of definition is a field of moduli.}
\item{The absolute field of moduli of a cover over $\algQ$ is the intersection of all fields of definition. In particular, if there exists a minimal (by inclusion) field of definition it coincides with the field of moduli.}
\end{enumerate}\end{proposition}

The first part of Proposition \ref{prop2} implies that the absolute field of moduli is a finite extension of $K$. Indeed every cover over $F$ can be defined over a finite subextension $L\supset K$ and the absolute field of moduli is contained in $L$.  

For some types of covers a field of moduli is a field of definition:
\begin{proposition}\label{prop3}Let $B$ be a curve over $K$. Consider a cover $f:X\to B_{F}$ with a field of moduli $M$.

\begin{enumerate}[label=(\roman*)]
\item If $\Aut(X,f)=\{\id\}$, then $(X,f)$ has a unique $M$-model.
\item If the cover $f$ is regular (i.e. $\Aut(X,f)$ acts transitively on the fiber of a non-critical $F$-point) and $B$ has a non-critical $M$-point $P$, $f$ is defined over $M$. Assume furthermore that $X$ has an $F$-point $Q$ over $P$. Then $f$ has a unique $M$-model $X_M\to B_M$ such that $X_M$ possesses an $M$-point over $P$.
\end{enumerate}
\end{proposition}

\begin{proof} (See also \cite{DD97}.)
Without loss of generality we may assume that $[F:K]<\infty$ since the infinite case reduces to the finite one as all covers over $F$ are defined over some finite subextension. 

\textit{(i)} If $f$ has no non-trivial automorphisms, there is a unique isomorphism of covers $g_*(X,f)\to (X,f)$ for each $g\in \Gal(F/M)$. It satisfies the cocycle condition \eqref{1} because of the uniqueness and therefore gives a unique $M$-model of $f$ by Proposition \ref{prop1}.

\textit{(ii)} Now suppose $f$ is regular and $P\in B(M)$ is non-critical. Choose an arbitrary geometric point $Q\in f^{-1}(P)$. Regularity of $f$ implies that for each $g\in \Gal(F/M)$ there exists a unique isomorphism of covers $\rho\colon g_*(X,f)\to (X,f)$ such that $\rho(g_*Q)=Q$. Similarly, this system of isomorphisms satisfies the cocycle condition \eqref{1} because of the uniqueness and defines an $M$-model of $f$ by Proposition \ref{prop1}. If furthermore $Q\in X(F)$, $(X,Q,f)$~is a pointed cover and there exists a unique isomorphism of pointed covers $g_*(X,Q,f)\to (X,Q,f)$ for each $g\in \Gal(F/M)$. Applying Proposition \ref{prop1} for pointed covers we obtain a unique pointed $M$-model $X_M\to B_M$.

\end{proof}

Note that the last assumption of Proposition \ref{prop3}(ii) always holds if $F=\algK$ is algebraically closed. In fact, the condition that the point $P$ in Proposition \ref{prop3}(ii) is non-critical can be weakened, see \cite{SJ} for a more general result.

\begin{proposition}\label{prop4}Let $B$ be a curve over $K$. Consider a cover $f:X\to B_{F}$ with a field of moduli $M$. Then the following holds true:
\begin{enumerate}[label=(\roman*)]
    \item {The cover can be decomposed as $f=f_0\circ f^{\reg}$ where $f^{\reg}:X\to X/\Aut(X,f)$ is the canonical projection and $f_0:X/\Aut(X,f)\to B$ is uniquely defined; }
    \item{There exists a canonical $M$-model $X_0\to B$ of $f_0$;}
    \item{$M$ is a field of moduli of $f^{\reg}$, the initial cover $f$ is defined over $M$ iff $f^{\reg}$ has an $M$-model with the base $X_0$.}
\end{enumerate}\end{proposition}
        We call $(X_0,f_0)$ and $(X,f^{\reg})$ the \textit{reduced} and \textit{regular}
        part of the cover $(X,f)$, respectively.

This result is stated in \cite[Theorem 3.3]{DE99}. Although the authors provide a complete proof only for a similar statement about curves, the case of covers can be established in the same manner. Here is a sketch of the proof.  
\begin{proof}
We may restrict ourselves to the case of finite degree $[F:K]<\infty$ as in the proof of Proposition \ref{prop3}.

\begin{enumerate}[label=(\roman*)]

\item Obvious.
\item Given an element $g\in \Gal(F/M)$ we want to consider an isomorphism of covers $$\tilde{\rho}(g):g_*(X,f)\to (X,f).$$ It induces an isomorphism of the corresponding reduced covers $$\rho(g):g_*(X/\Aut(X,f))\to X/\Aut(X,f).$$ It is straightforward to check that this isomorphism is independent on the choice of $\tilde{\rho}(g)$ and therefore meets the cocycle condition of Proposition \ref{prop1}.
\item It is clear from the proof of (ii) that descent data on $(X,f)$ uniquely correspond to descent data on $(X,f^{\reg})$. Thus the result follows from Proposition \ref{prop1}.
\end{enumerate}
\end{proof}

\begin{corollary} \label{cor1}Using the notations introduced above, let $L\supset M$ be an extension of the field of moduli such that $X_0$ has a non-critical $L$-point. Then $f$ is defined over $L$.\end{corollary}

\begin{proof} It follows from Propositions \ref{prop4} (iii), \ref{prop3}(ii). \end{proof}

As a consequence of the general results above we can deduce that genus-zero covers are defined over a quadratic extension of their field of moduli \cite[Theorem 5]{C94}.
\begin{corollary}
Let $f:\bbP^1_{F}\to\bbP^1_{F}$ be a cover with a field of moduli $M$. Then there is an extension $L\supset M$ such that $f$ is defined over $L$ and $[L:M]\leq 2$.
\end{corollary}
\begin{proof} 
In this case $X_0=\bbP^1_F/\Aut(\bbP^1_F,f)$ is a curve of genus 0 over $M$ and therefore it has a point over some quadratic extension of $M$.
\end{proof}

\subsection {The obstruction}
In this section we construct a cohomological obstruction for a cover being defined over its field of moduli. We start with the following folklore construction:

Let $B$ be a curve over $K$. Consider a cover $f:X\to B_{F}$ with a field of moduli $M$.
Let $G_{f,M}$ be the set of pairs $(g,\rho)$ where $g$ is an element of $\Gal(F/M)$ and $\rho$ is an isomorphism of covers $g_*(X,f) \to (X,f)$. The group law on $G_{f,M}$ is defined as follows: $$(g,\rho_1)\cdot(h,\rho_2)=(gh,\rho_1\circ g_*(\rho_2)).$$
It can be easily checked that $G_{f,M}$ is indeed a group with respect to this operation. There is a natural projection $G_{f,M}\to \Gal(F/M)$ whose kernel coincides with the group of automorphisms of $f$. So there exists an exact sequence of groups 
\begin{equation}\label{sq}\begin{CD}
1 @>>> \Aut(X,f) @>>> G_{f,M} @>>> \Gal(F/M) @>>> 1
\end{CD}\end{equation}
\begin{proposition}\label{prop5} In the setting introduced above we assume furthermore that $[F:K]<\infty$. Then the sequence \eqref{sq} splits iff $(X,f)$ is defined over $M$.\end{proposition}
\begin{proof}
The map $G_{f,M} \to \Gal(F/M)$ has a section iff a family of cover isomorphisms $\rho(g)\colon g_*(X,f) \to (X,f)$ for each $g\in \Gal(F/M)$ can be chosen to fit the cocycle condition of Proposition \ref{prop1}.
\end{proof}
Proposition \ref{prop5} allows us to construct a cohomological obstruction for covers with an abelian group of automorphisms. For another cohomological approach in a more general context see \cite{DD97}.
\begin{corollary}\label{cor3} Let $B$ be a curve over $K$. Consider a cover $f:X\to B_{F}$ with an abelian group of automorphisms. Then the obstruction for a field of moduli $M$ to be a field of definition is an element $\sigma \in H^2(\Gal(F/M),\Aut(X,f))$. 
\end{corollary}
In the case of an infinite Galois extension $K\subset F$, Galois cohomology groups should be understood in the profinite sense.
\begin{proof}
If $K\subset F$ is finite it is a standard fact that group extensions with abelian kernels are classified by the second cohomology group, see \cite[Theorem 3.12]{B82}. In the infinite case the statement follows from the definition of profinite cohomology groups $$H^2(\Gal(F/M),\Aut(X,f)):=\mathop{\lim_{\longleftarrow}}_{L} H^2(\Gal(L/M),\Aut_L(X,f)).$$
where the limit is taken over all finite subextensions $L\supset K$.
\end{proof}
\begin{remark} Let $L\supset M$ be a larger field of moduli of the cover $f:X\to B_{F}$. Then $G_{f,L}$ coincides with the preimage of $\Gal(F/L)\subset \Gal(F/M)$ under the projection $G_{f,M} \to \Gal(F/M)$. In other words, we have a commutative diagram
\begin{equation*}
\begin{CD}
1 @>>> \Aut(X,f) @>>> G_{f,L} @>>> \Gal(F/L) @>>> 1\\
@. @| @VVV @VVV @. \\
1 @>>> \Aut(X,f) @>>> G_{f,M} @>>> \Gal(F/M) @>>> 1
\end{CD}\end{equation*}
where the rows are exact and all squares are Cartesian. Hence, the group extension in the first row is the pullback of the group extension in the second row under the map $\Gal(F/L)\to \Gal(F/M)$. If the group of automorphisms $\Aut(X,f)$ is abelian, this property implies that the obstruction $\sigma\in H^2(\Gal(F/M),\Aut(X,f))$ constructed in Corollary \ref{cor3} is functorial in the following sense: the image of $\sigma$ under $$\res^M_L:H^2(\Gal(F/M),\Aut(X,f))\to H^2(\Gal(F/L),\Aut(X,f))$$ is precisely the obstruction corresponding to the field of moduli $L$. In particular, the cover $f$ is defined over $L$ iff $\res^M_L(\sigma)=0$.
\end{remark}

\section {Covers of genus one}
By \textit{a cover of genus one} we understand a cover $X\to \bbP^1$ where $X$ is a curve of genus one. 
In this section we focus on the absolute case $F=\algK \supset K$.
\newline
\textbf{Notation.} $G_L=\Gal (\algK/L)$ is the absolute Galois group of a field $L$.

\begin{proposition}\label{prop7} Let $f:X \to \bbP^1_{\algK}$ be a cover of genus one with a field of moduli $M$. Recall that $X_0=X/\Aut(X,f)$ is the canonical model of the reduced part of $f$ from Proposition \ref{prop4}. Then exactly one of the following holds true:
\begin{enumerate}
    \item {Either $\Aut(X,f)$ acts freely on $X$ and $X_0$ has genus one (in this case $\Aut(X,f)$ may be viewed as a finite subgroup of the elliptic curve $X$ with a properly chosen basepoint); Or}
    \item {$X_0$ has genus zero and $(X,f)$ is defined over a quadratic extension of $M$.}
\end{enumerate}\end{proposition}

\begin{proof} Let $d$ be the degree of $f^{\reg}$, let $e_P$ be the ramification index of $f^{\reg}$ at a point $P$ and let $g(X_0)$ stand for the genus. Then by Riemann-Hurwitz formula for $f^{\reg}$ $$d(2-2g(X_0))=\sum_{P\in X}(e_P-1)\geq 0.$$ Hence we deduce $g(X_0)\in \{0,1\}$. The cover $f^{\reg}$ is unramified (i.e $e_P=1$ at each $P\in X$) iff $\Aut(X,f)$ acts on $X$ without fixed points i.e. freely. 

If $X_0$ has genus one, the cover is a surjective morphism of genus-one curves which can be viewed as an isogeny of elliptic curves with properly chosen basepoints. Then $\Aut(X,f)$ is the kernel of the corresponding isogeny.  

Otherwise the cover is ramified and $X_0$ has genus zero. Then it has a point over a quadratic extension of the field of moduli and the statement follows from Corollary \ref{cor1} of Proposition \ref{prop4}.\end{proof}

As we are looking for covers with an arbitrary large degree of a field of definition over the field of moduli, we may restrict ourselves to the first case of Proposition \ref{prop7}.
In the rest of this section we show how the first case of Proposition \ref{prop7} can be reduced to the lifting of some cohomological class. 

Let $C$ be a curve over $\algK$. In order to define an action of the Galois group $\Gal(\algK/M)$ on the category of covers of $C$ we need to fix an $M$-model of $C$ and the resulting action depends on the chosen model. Therefore, the absolute field of modui of a cover $f:X\to C$ also depends on the chosen $M$-model of $C$. However, we prove that the absolute field of moduli of an isogeny of elliptic curves does not change when the base curve is replaced by any of its torsors.

\begin{lemma}\label{lemma0}
Let $E_0$ be an elliptic curve over a field $M$ and let $X_0$ be an $M$-torsor of $E_0$. Let us fix an arbitrary isomorphism of torsors $i\colon (E_0)_{\overline{K}}\iso (X_0)_{\overline{K}}$. Consider an isogeny $\phi:E\to (E_0)_{\overline{K}}$. Then the absolute field of moduli of the cover $\phi:E\to (E_0)_{\overline{K}}$ coincides with the absolute field of moduli of the cover $i\circ \phi:E\to (X_0)_{\overline{K}}$.  
\end{lemma}
\begin{proof} We denote the descent datum on $E_0$ by $\rho_{E_0}$. After the identification $i\colon (E_0)_{\overline{K}}\iso (X_0)_{\overline{K}}$, the descent datum on $E_0$ corresponding to the $M$-torsor $X_0$ has the form $\rho_{X_0}(g)=\rho_{E_0}(g)+\tau_0(g)$ where $\tau_0\in H^1(\Gal(\overline{K}/M),E_0)$.
Let us consider an element $g\in \Gal(\overline{K}/M)$. It is sufficient to demonstrate that $g_*(E,\phi)$ and $(E,\phi)$ are isomorphic as covers with the $M$-base $E_0$ iff they are isomorphic as covers with the $M$-base $X_0$. The first condition is equivalent to the existence of a suitable isomorphism $\rho:g_*E\to E$ such that the following diagram is commutative:
\begin{equation}\label{3}
    \xymatrixcolsep{5pc} \xymatrix{
g_*E \ar[d]_{g_*\phi} \ar[r]^{\rho} &E\ar[d]^\phi\\
g_*E_0 \ar[r]_{\rho_{E_0}(g)} &E_0}
\end{equation} Similarly the second condition is equivalent to the existence of another suitable isomorphism $\tilde{\rho}:g_*E\to E$ such that the diagram below is commutative: \begin{equation}\label{4} \xymatrixcolsep{5pc} \xymatrix{
g_*E \ar[d]_{g_*\phi} \ar[r]^{\tilde{\rho}} &E\ar[d]^\phi\\
g_*E_0 \ar[r]_{\rho_{E_0}(g)+\tau_0(g)} &E_0}
\end{equation} Let us pick some preimage $\tau$ of $\tau_0(g)$ under the isogeny $\phi$. Then one easily verifies that given any $\rho$ fitting into (\ref{3}), $\tilde{\rho} = \rho+\tau$ fits into (\ref{4}): 
$$\phi\circ\tilde\rho =\phi \circ (\rho +\tau)=\phi\circ\rho+\phi(\tau)=\phi\circ\rho+\tau_0(g)=\rho_{E_0}(g)\circ g_*\phi+\tau_0(g).$$
Similarly, given $\tilde{\rho}$ fitting into (\ref{4}), $\rho = \tilde{\rho}-\tau$ fits into (\ref{3}):
$$\phi\circ\rho =\phi \circ (\tilde\rho -\tau)=\phi\circ\tilde\rho-\tau_0(g)=\rho_{E_0}(g)\circ g_*\phi+\tau_0(g)-\tau_0(g)=\rho_{E_0}(g)\circ g_*\phi.$$
\end{proof}
Now we prove a technical lemma.
\begin{lemma}\label{lemma1}Let $X_0$ be a curve of genus one over a field $M$, and consider a finite map of genus-one curves $g:X\to (X_0)_{\overline{K}}$. Suppose $M$ is its field of moduli. Denote the Jacobian of $X_0$ by $E_0$. Then there exist an elliptic curve $E$ over $M$ and an isogeny $\phi:E\to E_0$ defined over $M$ which is isomorphic to $g$ over $\overline{K}$ after an identification of $(X_0)_{\overline{K}}$ and $(E_0)_{\overline{K}}$, the curve $E$ and the isogeny $\phi$ are uniquely determined. Moreover, if $(X,g)$ has an $M$-model $X_M \to X_0$, then $E$ is $M$-isomorphic to the Jacobian of $X_M$. \end{lemma}

\begin{proof}
From Lemma \ref{lemma0} we deduce that $M$ is a field of moduli of $\phi$ with respect to the $M$-base $E_0$.
Then the first claim follows from Proposition  \ref{prop3}(ii).

Let $g_M:X_M\to X_0$ be an $M$-model, we may consider the induced map on Jacobians $J(X_0)\to J(X_M)$. It is defined over $M$. Then the dual isogeny $J(X_M)\to J(X_0)=E_0$ is also defined over $M$ and is isomorphic to $(X,g)$ over $\overline{K}$. From the uniqueness part of the first claim it follows that $J(X_M)$ is isomorphic to $E$ over $M$.
\end{proof}

Let $f:X \to \bbP^1_{\algK}$ be a cover of genus 1 with a field of moduli $M$. We assume  $\Aut(X,f)$ acts freely on $X$. The genus-one curve $X_0$ is defined over $M$. Now we can apply Lemma \ref{lemma1} to obtain an $M$-isogeny of elliptic curves $\phi: E\to E_0=J(X_0)$ which is isomorphic to $f^{\reg}$ over $\algK$. Then there is an exact triple of $G_M$-modules:
$$
\xymatrix{
  0 \ar[r] & \Aut(f)  \ar[r] & E  \ar[r]^-{\phi} & E_0  \ar[r] & 0. }$$

It induces an exact sequence of Galois cohomology groups

$$
\xymatrix{
  H^1(G_M,E) \ar[r] & H^1(G_M,E_0)  \ar[r]^-{\delta} & H^2(G_M,\Aut(f)). }$$

\begin{proposition}\label{prop8}  In the setting introduced above there is a canonically defined class 
$\tau_0 \in H^1(G_M,E_0)$. The cover is defined over $M$ iff $\tau_0$ can be lifted to $\tau\in H^1(G_M,E)$. The image of $\tau_0$ under the boundary map $\delta$ coincides with the obstruction of Corollary \ref{cor3}. \end{proposition}

\begin{proof}
The class $\tau_0 \in H^1(G_M,E_0)$ corresponds to $X_0$ which forms a principal homogeneous space of its Jacobian $J(X_0)=E_0$. Let $\rho_E, \rho_{E_0}$ be the descend data corresponding to $E,E_0$. Then $X_0$ is defined by the descend datum $\rho_{E_0}+\tau_0$. If the cover $f^{\reg}$ has an $M$-model $X_M\to X_0$ we already know that $E$ is the Jacobian of $X_M$ (Lemma \ref{lemma1}), so $X_M$ is a torsor of $E$ defined by a descend datum of the form $\rho_E+\tau$ where $\tau\in H^1(G_M,E)$. The cover $f^{\reg}$ should be compatible with the descent data. In other words the following diagram is commutative:
$$\xymatrixcolsep{5pc} \xymatrix{
g_*E \ar[d]_{g_*\phi} \ar[r]^{\rho_E(g) +\tau(g)} &E\ar[d]^\phi\\
g_*E_0 \ar[r]_{\rho_{E_0}(g)+\tau_0(g)} &E_0}
$$
 This means that the following equality holds for any $x\in g_*E, g\in G_M$: 
 \begin{equation}\label{eq5}
\phi(\rho_E(g)(x)+\tau(g))=\rho_{E_0}(g_*\phi(x))+\tau_0(g). \end{equation}
Since $\phi:E\to E_0$ is defined over $M$, it is compatible with $\rho_E,\rho_{E_0}$ i.e. \begin{equation}\label{eq6}
    \phi\circ\rho_E(g)=\rho_{E_0}\circ g_*\phi.\end{equation}
Hence \eqref{eq5} is equivalent to the equality $\phi(\tau(g))=\tau_0(g)$. The latter equality implies that $\tau_0$ is the image of $\tau$ under the map $H^1(G_M,E) \to H^1(G_M,E_0)$.

The image of $\tau_0$ in $H^2(G_M,\Aut(f))$ corresponds to the exact sequence $$
\xymatrix{
  0 \ar[r] & \Aut(f)  \ar[r] & G_M\times_{\tau_0} E  \ar[r] & G_M \ar[r] & 0 }$$
  where $ G_M\times_{\tau_0} E=\{(g,p)\in G_M\times E|\tau_0(g)=\phi(p)\}$. We need to show that this group coincides with $G_{f,M}$ from the exact sequence \eqref{sq}. It suffices to verify that the condition $\tau_0(g)=\phi(p)$ is equivalent to the fact that $\rho_E(g)+p:g_*E\to E$ is an isomorphism of covers $g_*(E,\phi)\to(E,\phi)$. Indeed, $\rho_E(g)+p$ is an isomorphism of covers when the following compatibility condition holds for any $x\in g_*E$ (cf. the diagram above)
  \begin{equation*}
\phi(\rho_E(g)(x)+p)=\rho_{E_0}(g_*\phi(x))+\tau_0(g). \end{equation*}
From \eqref{eq6} we deduce that the latter is equivalent to $\phi(p)=\tau_0(g)$.
  \end{proof}

\section{The main construction } 
The aim of this section is to construct examples of Belyi functions such that the degree of a minimal field of definition over the field of moduli is large. In the previous section we have seen that the problem of defining a genus-one cover over the field of moduli reduces to some isogeny of elliptic curves. Here we consider the case of the simplest isogeny, namely, multiplication by $m$. First we show that under some assumptions on a finite subgroup $C$ of an elliptic curve $E$ any element in $H^2(G_M,C)$ can be obtained as the obstruction of Corollary \ref{cor3}. Then we demonstrate the existence of an appropriate element in $H^2(G_M,E[m])$ which does not split over extensions of small degree. 

The base field $K$ is assumed to be a number field, $K\subset F = \algQ$.
We need a fact from arithmetic geometry \cite[Corollary 6.24, Corollary 3.4]{M06}:
  \begin{proposition}\label{prop}
  For any elliptic curve $E$ over a number field $M$ there  is a functorial isomorphism
  $$H^2(G_M,E)\iso \prod_{\nu\; real}H^2(G_{M_\nu},E)$$ where $\nu$ runs through all real places of $M$. For real places $\nu$, we have $$H^2(G_{M_\nu},E)= \begin{cases} 
      \bbZ/2\bbZ & \text{if} \; \Delta(E)> 0 \\
      0 & \text{if}\; \Delta(E)<0 
   \end{cases}$$ \end{proposition}
\begin{corollary}\label{lemma2}
Consider an $M$-isogeny $\phi:E\to E_0$ of elliptic curves over $M$ such that $\deg(\phi)$ is odd. Then it induces an isomorphism  $$H^2(G_M,E)\to H^2(G_M,E_0).$$
\end{corollary}
\begin{proof}
  It is sufficient to show that $\phi$ induces an isomorphism $$H^2(G_{M_\nu},E)\to H^2(G_{M_\nu},E_0)$$ at each real place $\nu$. Denote the dual isogeny by $\hat{\phi}:E_0\to E$. The composition $\hat{\phi}\circ \phi$ induces a composite map $$H^2(G_{M_\nu},E)\to H^2(G_{M_\nu},E_0)\to H^2(G_{M_\nu},E)$$ which is the multiplication by an odd integer $\deg(\phi)^2$. Therefore it is an isomorphism. Similarly, the other composition $\phi\circ \hat{\phi}$ induces an isomorphism. Then the first arrow is also an isomorphism since all the cohomology groups are either trivial or cyclic of order 2.
 
\end{proof}

Recall that \textit{a Belyi function} is a surjective morphism $\beta: X\to \bbP^1$ unramified outside $\{0,1,\infty\}$, the cover $(X,\beta)$ is also called \textit{a Belyi pair}.
We quote a refined version of Belyi theorem from \cite[p.2, Th\'eor\`eme]{C96} which we use to prove the key lemma below.

  \begin{theorem}\label{th1}
  For each curve $X$ over a number field $M$, there exists a Belyi function $\beta: X\to \bbP^1_M$ over $M$ such that $\Aut(X,\beta)=\{id\}$. 
  \end{theorem}
  
  Recall that a \textit{non-CM elliptic curve} is an elliptic curve $E$ which does not have a complex multiplication or  equivalently $\End(E)=\bbZ$. A \textit{$G_M$-submodule} of an elliptic curve $E$ over a field $M$ is a subgroup $C\subset E$ which is preserved by the action of $G_M$, that is $gC\subset C$ for all $g\in G_M$.
  
\begin{lemma}\label{lemma3}
Let $E$ be a non-CM elliptic curve over a number field $M$ and let $C\subset E$ be a finite subgroup such that $C$ is a $G_M$-submodule of $E$. Assume that at least one of the following holds true: \begin{enumerate}[label=(\roman*)]
    \item The order of $C$ is odd;
    \item The field $M$ is purely imaginary.
\end{enumerate}Then for each element $\sigma \in H^2(G_M,C)$ there exists a Belyi map $\beta:E_{\algQ}\to \bbP^1_{\algQ}$ with a field of moduli $M$ having $\sigma$ as the obstruction constructed in Corollary \ref{cor3}.
\end{lemma}

\begin{proof}
Let us denote the quotient curve $E/C=E_0$. It is defined over $M$ because $C$ is a $G_M$-submodule of $E$. Then the exact sequence of $G_M$-modules $$
\xymatrix{
  0 \ar[r] & C  \ar[r] & E  \ar[r]^-{} & E_0  \ar[r] & 0 }$$ induces a long exact sequence 
  $$
\xymatrix{
  H^1(G_M,E) \ar[r] & H^1(G_M,E_0)  \ar[r]^-{\delta} & H^2(G_M,C)  \ar[r] & H^2(G_M,E)  \ar[r] &  H^2(G_M,E_0)}.$$
  In the case of an odd order of $C$ the latter arrow is an isomorphism by Corollary \ref{lemma2}. Therefore the boundary map $\delta$ is surjective. If $M$ is purely imaginary $H^2(G_M,E)=0$ by  Proposition \ref{prop}, so $\delta$ is also surjective.
  
  Choose $\sigma \in H^2(G_M,C)$ and lift it to an element $\tau_0 \in H^1(G_M,E_0)$. Then the element $\tau_0$ defines a principal homogeneous space $X_0$ of the elliptic curve $E_0$ over $M$. 
  
  Applying Theorem \ref{th1} to $X_0$ we obtain a Belyi map $\beta_0: X_0\to \bbP^1_{M}$ without non-trivial automorphisms. We next want to consider the composition $\beta=(\beta_0)_{\algQ}\circ (\phi)_{\algQ}$ where $\phi$ is the isogeny $E\to E_0$: 
 $$\beta: \xymatrix{
  E_{\algQ} \ar[r]^-{(\phi)_{\algQ}} &  (E_0)_{\algQ} \ar[r]^{\sim} & (X_0)_{\algQ} \ar[r]^{(\beta_0)_{\algQ}} &  \bbP^1_{\algQ}}.$$
 
  (We identify $(E_0)_{\algQ}$ with $(X_0)_{\algQ}$.) First note that $\beta$ is a Belyi function over $\algQ$ because $\phi$ is an unramified cover. 
  
  Let $h:E\to E$ be an automorphism of $\beta$. Since $E$ is non-CM, $h$ has the form $h(x)=\varepsilon \cdot x+a$ where $\varepsilon\in \{+1,-1\}$, $a\in E$. Then we deduce $$\beta_0( \phi (x))=\beta_0( \phi (\varepsilon x+a))=\beta_0(\varepsilon\phi(x)+\phi(a)). $$ Therefore $\varepsilon \cdot x+\phi(a)$ is an automorphism of $\beta_0$. Consequently, $\varepsilon=1, \phi(a)=0$ since $\beta_0$ is automorphisms free. Hence $\Aut(X,\beta)=\ker(\phi)$ and $\beta_0$ is the reduced part of $\beta$. 
  
  To see that $\beta$ fits our requirements we only have to check that $M$ is a field of moduli of $\beta$. Since $M$ is a field of moduli (and even a field of definition) of $\phi$, Lemma \ref{lemma0} implies that $M$ is a field of moduli of $\phi$ with respect to the $M$-base $X_0$. Therefore $M$ is a field of moduli of $\beta$ since $\beta_0:X_0\to\bbP^1_{M}$ is defined over $M$.
  
  \end{proof}
  \begin{proposition}\label{prop9}
  Consider an elliptic curve $E$ over a number field $M$. Assume that $E[m]\subset E(M)$ and $\zeta_m\in M$ for a positive integer $m$ where $\zeta_m$ is a primitive $m$-th root of unity. Then there exists an element $\sigma \in H^2(G_M,E[m])$ such that $\res^M_L(\sigma)\neq 0$ for all $L\supset M$ with $[L:M]<m$.
  \end{proposition}
\begin{proof}
Under the hypothesis of the Proposition we have an isomorphism of $G_M$-modules $$E[m]\cong \mu_m \times \mu_m$$ as both of them have trivial Galois action. In particular, $$H^2(G_M,E[m])\cong H^2(G_M,\mu_m)\times H^2(G_M,\mu_m).$$
The Kummer exact sequence $$
\xymatrix{
  1 \ar[r] & \mu_m  \ar[r] & \algK^* \ar[r]^-{m} & \algK^*  \ar[r] & 1}$$  together with Hilbert 90 theorem lead to a canonical isomorphism $$H^2(G_M,\mu_m)\iso Br(M)[m].$$ Therefore it is sufficient to find an element in $Br(M)[m]$ which does not split over extensions of degree less than $m$. For example, one can take a cyclic algebra of order $m$ \cite{Mcft}.  
\end{proof}
We are now ready to prove our first main result. 
\begin{theorem}\label{th2} Let $E$ be a non-CM elliptic curve over $\algQ$ and let $m$ be a positive integer. Then there exists a Belyi function $\beta:E \to \bbP^1_{\algQ}$  with the absolute field of moduli $M^{abs}$ and the following property:
for any field of definition $L$ of the Belyi pair $(E,\beta)$, we have $[L:M^{abs}]\geq m$.
\end{theorem}
\begin{proof}
We may assume $m$ is odd. The curve $E$ is defined over some number field $K$. Consider the field $M=K(E[m], \zeta_m)$. Apply Proposition \ref{prop9} to get an element $\sigma \in H^2(G_M,E[m])$ which does not split over extensions of degree less than $m$. According to Lemma \ref{lemma3}, $\sigma$ arises as an obstruction for some Belyi function $\beta: E_{\algQ} \to \bbP^1_{\algQ}$. Let $M^{abs}$ be the absolute field of moduli of $\beta$. Then $M^{abs}$ is contained in $M$ by definition. If $\beta$ is defined over an extension $L\supset M^{abs}$, then $\sigma$ vanish when restricted to $ML$. Therefore we conclude $$[L:M^{abs}]\geq [ML:M]\geq m.$$

\end{proof}

\section{A counterexample to the Hasse principle}
Using the cohomological approach from the previous section one may construct a Belyi pair defined everywhere locally though not globally. 

\begin{definition}
 Let $M$ be a number field and let $W$ be a $G_M$-module. We define the \textit{$i$-th Tate-Shafarevich group of $W$} as the kernel of the corresponding local-global map:
 $$\Sh^i(G_M,W)=\ker\left(H^i(G_M,W)\to\prod_{\nu}H^i(G_{M_\nu},W)\right)$$ where the product is taken for all places $\nu$.
\end{definition}

\begin{theorem}\label{th3}
Let $M$ be a number field, let $E$ be a non-CM elliptic curve over $M$ and let $m$ be a positive integer such that $m$ is odd or $M$ is purely imaginary. Additionally assume that $\Sh^1(G_M,E[m])\neq 0$. Then there exists a Belyi function $\beta:E_{\algQ}\to\bbP^1_{\algQ}$ with a field of moduli $M$ such that $\beta$ is defined over $M_\nu$ for all places $\nu$ and $\beta$ is not defined over $M$.
    
\end{theorem}
\begin{proof}
By global Tate duality \cite[Theorem 4.10]{M06} $\Sh^2(G_M,E[m])$ is dual to $\Sh^1(G_M,E[m])$ and therefore it is non-trivial. Let us pick a non-zero element $\sigma \in \Sh^2(G_M,E[m])$. By Lemma \ref{lemma3} it arises as an obstruction for some Belyi function $\beta: E_{\algQ}\to\bbP^1_{\algQ}$. Note that $\res_\nu(\sigma)=0$ is the obstruction corresponding to the local cover $(E_{\overline{K_\nu}}, \beta_{\overline{K_\nu}})$. So $\beta$ is defined over $M_\nu$ for all places $\nu$ however it is not defined over $M$. \end{proof}

\textbf{A numerical example.} We present an example when the conditions of Theorem \ref{th3} are satisfied. The exact sequence $$
\xymatrix{
  0 \ar[r] & E[m]  \ar[r] & E  \ar[r]^-{m} & E  \ar[r] & 0}$$ induces a canonical injection $$H^0(G_M,E/mE)=E(M)/mE(M)\into H^1(G_M,E[m]).$$ Therefore $\Sh^0(G_M,E/mE)\neq 0$ would imply $\Sh^1(G_M,E[m])\neq 0$. In \cite{DZ04} the authors provide the following example: a non-CM elliptic curve $$E:y^2=(x+2795)(x-1365)(x-1430)$$ has a point $$P=\left(\frac{5086347841}{1848^2},-\frac{35496193060511}{1848^3}\right)\in E(\bbQ)$$ such that $P$ is divisible by $4$ at all places $\nu$ of $\bbQ$ but $P$ cannot be divided by $4$ over $\bbQ$. In other words $P$ gives a non-trivial class $[P]\in \Sh^0(G_{\bbQ},E/4E)$. Note that $\res^{\bbQ}_{M}[P]\neq 0$ for some purely imaginary number field $M$. For instance, one can take any imaginary $M$ linearly disjoint from $\bbQ(P/4)$. Then the triple $(M,E,4)$ fits the requirements.
  
  \begin{corollary}\label{cor4}
  There exists a genus-one Belyi pair $(X,f)$ over $\algQ$ which cannot be defined over its field of moduli $M$, yet which can be defined over the completion $M_\nu$ for all places $\nu$ of $M$.
  \end{corollary}

{\footnotesize{\scshape Département de mathématiques et applications, École normale supérieure, \\45 rue d’Ulm, Paris, France 75005}}
\\
\textit{Email address:} amolyakov1@gmail.com

\end{document}